\documentclass[12pt]{amsart}
\usepackage{amsmath,amsthm,amssymb,amstext,amsxtra}
\usepackage[all]{xy}
\usepackage{mathrsfs}  
\usepackage{mathtools} 
\usepackage{colonequals}
\usepackage{enumerate}

\usepackage[backref, colorlinks=true, linkcolor=blue, citecolor=blue,urlcolor=blue]{hyperref}
\usepackage{fullpage}

\usepackage[nameinlink]{cleveref}

\newcommand{\linedef}[1]{\textsf{#1}}
\newcommand{\defi}[1]{\textit{\linedef{#1}}}

\newcommand \C {\mathbb C}

\newcommand \F {\mathbb F}

\renewcommand \O {\mathscr O}

\newcommand \Q {\mathbb Q}

\newcommand\Zl{\mathbb{Z}_{\ell}}

\newcommand \R {\mathbb R}

\newcommand \Z {\mathbb Z}

\newcommand \calI {\mathcal I}
\newcommand \calT {\mathcal T}
\newcommand \Kbar {K^{\textup{al}}}
\newcommand \Qbar {\Q^{\textup{al}}}

\newcommand \rhobar {\overline{\rho}}
\newcommand \Zhat {\widehat{\Z}}
\newcommand \Zhatp {\widehat{\Z}'}
\newcommand \Qhatp {\widehat{\Q}'}

\newcommand{\Lbun}{\mathscr{L}}
\newcommand{\scrO}{\mathscr{O}}
\newcommand{\scrL}{\Lbun}

\DeclareMathOperator{\coker}{coker}
\DeclareMathOperator{\ddiv}{div}

\DeclareMathOperator \Gal {Gal}
\DeclareMathOperator \GL {GL}

\DeclareMathOperator \Pic {Pic}

\DeclareMathOperator \Sym {Sym}

\DeclareMathOperator \Aut {Aut}

\DeclareMathOperator \Ob {Ob}
\DeclareMathOperator \opchar {char}

\DeclareMathOperator \Hom {{Hom}}
\DeclareMathOperator \sep {{s}}
\DeclareMathOperator \bPic {\mathbf{Pic}}
\newcommand \kclosed {k^{\sep}}
\newcommand \kal {k^{\textup{al}}}

\newcommand\cosl{\calI^{A_0}}

\newcommand\bT{\widehat{T}}
\newcommand\bV{\widehat{V}}

\newcommand\TAze{\calT_{A_0}}

\DeclareMathOperator{\id}{id}

\DeclareMathOperator{\impart}{Im}

\numberwithin{equation}{subsection}

\theoremstyle{plain}

\newtheorem{thm}[equation]{Theorem}

\newtheorem{prop}[equation]{Proposition}

\newtheorem{lem}[equation]{Lemma}

\theoremstyle{definition}
\newtheorem{defin}[equation]{Definition}
\newtheorem{example}[equation]{Example}
\newtheorem{exm}[equation]{Example}

\theoremstyle{remark}
\newtheorem{remark}[equation]{Remark}
\newtheorem{rmk}[equation]{Remark}

\newenvironment{enumalph}
{\begin{enumerate}}
{\end{enumerate}}

\title{A framework for Tate modules of abelian varieties under isogeny}

\author{Sarah Frei}
\address{Department of Mathematics, Dartmouth College, Kemeny Hall, Hanover, NH 03755, USA}
\email{sarah.frei@dartmouth.edu}

\author{Katrina Honigs}
\address{Department of Mathematics, Simon Fraser University, 8888 University Drive, Burnaby, British Columbia V5A 1S6, Canada}
\email{khonigs@sfu.ca}

\author{John Voight}
\address{Department of Mathematics, Dartmouth College, Kemeny Hall, Hanover, NH 03755, USA; Carslaw Building (F07), Department of Mathematics and Statistics, University of Sydney, NSW 2006, Australia}
\email{jvoight@gmail.com}

\date{\today}



\begin{document}

\begin{abstract}
We explain the linear algebraic framework provided by Tate modules of isogenous abelian varieties in a category-theoretic way.
\end{abstract}

\maketitle

\section{Introduction}

\subsection{Setup}

Let $K$ be a number field with algebraic closure $\Kbar$.  Let $A$ be an abelian variety over $K$ of dimension $g \colonequals \dim A \geq 1$. For example, we may take $A=E$ an elliptic curve over $K$, the case $g=1$.  
Many important arithmetic features of $A$ are reflected in its torsion subgroups $A[n](\Kbar) \simeq (\Z/n\Z)^{2g}$ for $n \geq 1$.  The equations that define the $n$-torsion subgroup define a variety $A[n]$ of dimension zero over $K$; and the finite set of points $A[n](\Kbar)$ is defined over its splitting field, a minimal finite Galois extension of $K$ denoted $K(A[n])$.  The Galois group $\Gal(K(A[n])\,|\,K)$ acts on $A[n]$ preserving the group law, so each element acts via an element of $\Aut(A[n](\Kbar))$, giving an injective homomorphism
\[ \Gal(K(A[n])\,|\,K) \hookrightarrow \Aut(A[n](\Kbar)) \simeq \GL_{2g}(\Z/n\Z). \]
It is convenient to lift this to the absolute Galois group $\Gal_K \colonequals \Gal(\Kbar\,|\,K)$ to obtain a linear representation
\begin{equation} \label{eqn:rhoAbarn}
\rhobar_{A,n} \colon \Gal_K \to \Aut(A[n](\Kbar)) \simeq \GL_{2g}(\Z/n\Z)
\end{equation}
with $\ker \rhobar_{A,n} = \Gal(\Kbar\,|\,K(A[n]))$.  Studying the Galois representation $\rhobar_{A,n}$---for example, using techniques in number theory, group theory, and linear algebra---remains an essential technique for understanding $A$ and is itself an interesting pursuit.  

Rather than working with one $n$ at a time, it is convenient to package all of them together by forming the adelic Tate module 
\begin{equation}
\widehat{T} A \colonequals \varprojlim_n A[n](\Kbar) \simeq \Zhat^{2g} 
\end{equation}
where $\Zhat \colonequals \varprojlim_n \Z/n\Z \simeq \prod_p \Z_p$, and the associated Galois representation 
\begin{equation}
\rho_A \colon \Gal_K \to \Aut(\widehat{T} A) \simeq \GL_{2g}(\Zhat). 
\end{equation}
One obtains the representations $\rhobar_{A,n}$ by composing $\rho_A$ with reduction modulo $n$.  Allowing denominators, we obtain the adelic Tate representation $\widehat{V} A \colonequals \widehat{T} A \otimes \Q \simeq \widehat{\Q}^{2g}$.  

In this article, we unpack how the adelic Tate module (with its Galois action) changes under isogeny.  This is a well-known, fundamental tool that is used frequently in arithmetic geometry.  For example, in his work on local-global principles for torsion, Katz \cite[p.~482]{Katz} calls this the ``dictionary between $A'\hspace*{.03cm}$'s which are $l$-power-isogenous to $A$ over $K$ and $\Gal(\overline{K}/K)$-stable \emph{lattices} in $T_l(A) \otimes \Q$''.  This dictionary as a bijection is also stated by Lan \cite[\S 1.3.5]{MR3186092}, again without proof.  Here, we formulate this dictionary within a category-theoretic framework, so that future matrix calculations with Tate modules can be understood via commutative diagrams.  

\subsection{Results}

Let $A_0$ be a (fixed) abelian variety over $K$.  
Let $\varphi \colon A_0 \to A$ be an isogeny over $\Kbar$.  Then $\varphi$ induces a natural inclusion $\widehat{T}\varphi \colon \widehat{T} A_0 \hookrightarrow \widehat{T} A$ of adelic Tate modules, which becomes an isomorphism $\widehat{V}\varphi \colon \widehat{V} A_0 \xrightarrow{\sim} \widehat{V} A$.  Composing the natural inclusion $\widehat{T} A \hookrightarrow \widehat{V} A$ with $(\widehat{V}\varphi)^{-1}$, we obtain an inclusion $\widehat{T} A \hookrightarrow \widehat{V} A_0$ whose image is a sublattice denoted $\Lambda\varphi \subseteq \widehat{V} A_0$.  

Our first main result is as follows (proven in more generality in \Cref{thm:main-full}).  

\begin{thm}\label{thm:main}
The association $\varphi \mapsto \Lambda\varphi$ determines a functor $\Lambda$ from 
\begin{center}
the category of isogenies $\varphi \colon A_0 \to A$ over $\Kbar$
\end{center}
to 
\begin{center}
the category of sublattices of $\widehat{V}A_0$ containing $\widehat{T}A_0$.  
\end{center}
The functor $\Lambda$ has the following properties.
\begin{enumalph}
\item $\Lambda$ is an equivalence of categories.
\item The functor is equivariant with respect to the action of $\Gal_K$, restricting to an equivalence between the category of abelian varieties that are isogenous (over $K$) to $A_0$ and the category of $\Gal_K$-stable sublattices of $\widehat{V}A_0$.  
\item For an isogeny $\varphi \colon A_0 \to A$, there is a natural isomorphism $\ker \varphi \simeq \Lambda\varphi/\Lambda(\id_{A_0}) = \coker \bT \varphi$ that is $\Gal_K$-equivariant.
\end{enumalph}
\end{thm}

The morphisms in the two categories in \Cref{thm:main} are described in \cref{tw.on}.  

With this conceptual framework in mind, we then turn to some computations, explaining how to extract explicit change of basis matrices that explain how isogenous abelian varieties have conjugate Galois representations; we similarly track how polarizations change.  
We provide an extended example that came up in recent work \cite{FHV_research}, giving an example of abelian surfaces $A$ such that the $\ell$-torsion of $A$ and its dual $A\spcheck$ are not isomorphic as Galois representations, i.e., $\rho_{A,\ell} \not\simeq \rho_{A\spcheck,\ell}$.
\subsection{Contents}

We begin by providing background on abelian varieties and relevant structures such as polarizations and duality in \cref{sec:background}.  
In \cref{sec:category} we explain (in a categorical context) how Galois actions on Tate modules change under isogenies, along with an illustrating example. We define the functor that satisfies \Cref{thm:main} in \cref{tw.on}.  We concretely work with this framework by providing change of basis matrices in \cref{sec:changeofbasis}.  
In \cref{example.section}, we work through a detailed example elaborating on the main result of \cite{FHV_research}. 
\medskip

We hope that this will serve as a useful framework for others interested in studying Galois images of torsion subgroups and Tate modules of abelian varieties.

\subsection*{Acknowledgements}

The authors would like to thank Asher Auel, Nils Bruin, Ryan Chen, and Bjorn Poonen for helpful comments. Many thanks also to the anonymous referees for their feedback. Frei was supported by an AMS-Simons Travel grant and by NSF grant DMS-2401601. Honigs was supported by an NSERC Discovery grant.
Voight was supported by grants from the Simons Foundation: (550029, JV) and (SFI-MPS-Infra\-structure-00008650, JV).

\section{Background on abelian varieties}\label{sec:background}

In this section, we set notation and give background on
isogenies and Tate modules (\cref{subsec:isogenies}), dual abelian varieties (\cref{subsec:dual}), 
the Weil pairing (\cref{subsec:Weil}),
and polarizations (\cref{subsec:polar}).  For further reading, a standard reference for abelian varieties is the book of Mumford \cite{Mumford}; we also recommend Hindry--Silverman \cite[Part A]{HinSil}, the work of Milne \cite{Milneabvar,Milne2}, and Birkenhake--Lange \cite{Birkenhake-Lange} for complex abelian varieties.  To just get started with introduction via the perspective of elliptic curves, we suggest Silverman \cite{SilvermanAEC}.  We endeavour in this section to give a motivated version suitable for study in our formalism.  

A key result is \Cref{kercoker}; here, we provide a simple proof.

\subsection{Isogenies and Tate modules}\label{subsec:isogenies}

Let $k$ be a field and let $p \colonequals \opchar k$, allowing $p=0$.  Let $\kclosed \subseteq \kal$ be a separable closure of $k$ and an algebraic closure of $k$, respectively.  Let \(A\) and \(A'\) be abelian varieties over $k$ with common dimension $g \colonequals \dim A = \dim A'$.  To avoid trivialities, we suppose that $g \geq 1$.  

An \defi{isogeny} \(\varphi \colon A \to A'\) is a surjective homomorphism, or equivalently a homomorphism with finite kernel $\#(\ker \varphi)(\kal)<\infty$.  (Some authors take the zero map to be an isogeny; we do not.)  An isogeny is \defi{separable} if the corresponding finite extension $k(A) \supseteq k(A')$ of function fields is separable, or equivalently $(\ker \varphi)(\kal)=(\ker \varphi)(\kclosed)$ and $[k(A):k(A')]=\#(\ker \varphi)(\kclosed)$---analogous to the usual condition for a finite extension of fields to be Galois.  (The reader may wish to focus on the case $k=\Q$ where all isogenies are separable; but the setup permits an arbitrary field.)  

We pause to give a bit of motivation before proceeding further.  Over the complex numbers $k=\C$, abelian varieties are complex tori, and isogenies can be understood using linear algebra.  Indeed, we have an isomorphism $A(\C) \simeq V/\Lambda$ where $V \simeq \C^g$ is a complex vector space of dimension $g$ and $\Lambda \subseteq V$ is a lattice of rank $2g$.  (More precisely, we take $V = \Hom(\Omega^1(A),\C)$ dual to the space $\Omega^1(A)$ of holomorphic $1$-forms and $\Lambda = H_1(A,\Z)$.)  Then every isogeny of complex abelian varieties $V/\Lambda \to V'/\Lambda'$ is defined by a  $\C$-linear isomorphism $V \to V'$ such that $\varphi(\Lambda) \subseteq \Lambda'$, and  
\begin{equation} \label{eqn:kervarphi}
\ker \varphi = \varphi^{-1}(\Lambda')/\Lambda \cong \Lambda'/\varphi(\Lambda) = \coker(\phi \colon \Lambda \hookrightarrow \Lambda').
\end{equation}
This is such a convenient description!  We seek to replicate it over an arbitrary field $k$, and Tate modules allow us to consider (separable) isogenies in an analogous way.

We can begin linearizing, as in the introduction, by working with a torsion subgroup $A[n]$---but instead of working with just one, to complete the analogy we package them together, as follows.  

Let $\ell \neq p= \opchar k$ be prime.  The \defi{$\ell$-adic Tate module} of $A$ is the projective limit 
\begin{equation}
T_\ell A \colonequals \varprojlim_{j} A[\ell^j](\kclosed) \simeq \Z_\ell^{2g}.
\end{equation}
Concretely, an element of $T_\ell A$ is a sequence $P=(P_1,P_2,\dots)$ where $P_j \in A[\ell^j](\kclosed)$ and $\ell P_j = P_{j-1}$ for all $j \geq 2$.  

A homomorphism \(\varphi \colon A \to A'\) induces a homomorphism \(T_\ell(\varphi) \colon T_\ell A \to T_\ell A'\), and when \(\varphi\) is an isogeny, the finiteness of the kernel implies that \(T_\ell(\varphi)\) is injective.  

For complex abelian varieties (when $k=\C$), we recover in this way the $\ell$-adic part of the description above. We have 
\[ A[\ell^j](\C) = (\ell^{-j}\Lambda)/\Lambda \cong \Lambda/\ell^j \Lambda \]
for all $j$, so $T_\ell A \cong \Lambda \otimes \Z_\ell$ is the $\ell$-adic completion of $\Lambda$.  
An isogeny $\varphi \colon V/\Lambda \to V'/\Lambda'$ induces a map of $\ell$-adic Tate modules $T_\ell(\varphi) \colon \Lambda \otimes \Z_\ell \to \Lambda' \otimes \Z_\ell$ (applying the isogeny in each component, checking compatibility); and taking the $\ell$-primary subgroups in \eqref{eqn:kervarphi} we see that 
\begin{equation} \label{eqn:kerTell}
(\ker \varphi)(\C)_\ell \cong (\Lambda' \otimes \Z_\ell)/(\varphi(\Lambda) \otimes \Z_\ell) \cong \coker T_\ell(\varphi).
\end{equation}
The following proposition shows that \eqref{eqn:kerTell} extends in general.  
\begin{prop}\label{kercoker}
Let $\varphi\colon A\to A'$ be an isogeny of abelian varieties with kernel $H$.  Let $H[\ell^\infty]$ be the $\ell$-primary (or $\ell$-Sylow) subgroup of $H$.  Then $H[\ell^\infty]$ is naturally isomorphic to the cokernel of $T_\ell(\varphi)\colon T_\ell A\to T_\ell A'$.  
\end{prop}

\begin{proof}
  
We will show, via the snake lemma, that this map is induced by a connecting homomorphism.

Abbreviate $A_j \colonequals A[\ell^j](\Kbar)$.  For $n,r \geq 1$, the sequence
\[ 0 \to A_n \to A_{n+r} \xrightarrow{\cdot \ell^n} A_r \to 0 \]
is exact.  The maps are compatible, so for all $r \geq 1$ and taking the limit over $n$, we claim that the sequence
\[ 0 \to T_\ell A \to \varprojlim_{n} A_{n+r} \to A_r \to 0 \]
is exact.  The left map is just interpreting the same sequence in two different ways, which is visibly injective.  Indeed, projective limits are always left exact so it suffices to show that the map to $A_r$ is surjective.  This map takes $(P_n)_n$ with $P_n \in A_{n+r}$ and maps to the common element $\ell P_1 = \ell^n P_n$ for all $n$.  Hence, for any $Q=P_0 \in A_r$, a lift $(P_n)_n \in \varprojlim_n A_{n+r}$ is obtained by inductively choosing $P_n \in A_{n+r}$ satisfying $\ell P_n = P_{n-1}$ for $n \geq 1$.

Repeating the abbreviation with $A'$ and $H$, we put the exact sequences together to get:
\begin{equation}
\begin{aligned}
\xymatrix{
0 \ar[r] & T_\ell A \ar[r] \ar[d]^{\varphi} & \varprojlim_{n} A_{n+r} \ar[r]^(0.6){\cdot \ell^n} \ar[d]^{\varphi} & A_r \ar[r] \ar[d]^{\varphi} & 0 \\
0 \ar[r] & T_\ell A' \ar[r] & \varprojlim_{n} A'_{n+r} \ar[r]^(0.6){\cdot \ell^n} & A'_r \ar[r] & 0
} 
\end{aligned}
\end{equation}
We now apply the snake lemma.  The kernel of the middle vertical map is $0$ since $\ker \varphi$ is finite, so we obtain the exact sequence
\begin{equation} \label{eqn:iwontTell}
0 \to H_r \xrightarrow{\delta} T_\ell A'/\varphi(T_\ell A) \to \varprojlim_n A'_{n+r}/\varprojlim_n \varphi(A_{n+r}).
\end{equation}
To be explicit, the connecting map $\delta$ is defined as follows: for $P \in H_r$, lift to $(P_n)_n \in \varprojlim_n A_{n+r}$ (as in the end of the previous paragraph) and map 
\begin{equation} 
\delta(P)=\varphi((P_n)_n)=(\varphi(P_1),\varphi(P_2),\dots) \in T_\ell A'. 
\end{equation}
(We do not need to, but can check that $\ell^n \varphi(P_n)=\varphi(\ell^n P_n)=\varphi(P)=O$ so $\varphi(P_n) \in A'[\ell^n]$.)

Now we take $r \geq 1$ to be such that $\ell^r=\#H_\infty$; then $H_r=H_\infty$.  To finish, we need to check that the final map in \eqref{eqn:iwontTell} is the zero map.  Indeed, the class of $(Q_n)_n \in T_\ell A' \leq \prod_n A'_n$ maps to the class of $(Q_n)_n \in \prod_n A'_{n+r}$, so we want to show for all $n \geq 1$ and all $Q_n \in A'_n$ that $Q_n=\varphi(P_n')$ for some $P_n' \in A_{n+r}$.  But $\deg \varphi = \ell^r$ so there exists an isogeny $\psi \colon A' \to A$ such that $\ell^r=\varphi\psi$;
therefore, if $P_n \in A_{n+r}$ is such that $\ell^r P_n=Q_n$ then $P_n' \colonequals \psi(P_n)$ has $\varphi(P_n')=\ell^r P_n=Q_n$ and since $\ell^{n+r} P_n = \ell^n Q_n = O$ we have $\ell^{n+r} P_n' = \psi(\ell^{n+r} P_n) = O$ and so $P_n' \in A_{n+r}$ as desired.

Finally, we address naturality.
For any finite subgroup $H\leq H'\leq A$, there is an isogeny $A\to A''\simeq A/H'$ factoring uniquely through $A\to A'$. 
The isomorphism $H'[\ell^{\infty}]\simeq \coker(T_{\ell}A\to T_{\ell}A'')$
given by the above construction restricts to the isomorphism $H\simeq \coker(T_{\ell}A\to T_{\ell}A')$
because the connecting morphism from the snake lemma is compatible with this restriction. 
\end{proof}

\begin{example}\label{cyclic_isog}
Consider the case $g=1$, so $A=E$ is an elliptic curve.  Choose a basis  
$P_1=(P_{1,n})_n$, $P_2=(P_{2,n})_n$ 
for $T_\ell E$. 
Consider the isogeny $\varphi\colon E\to E'$
with kernel $H=H_\ell=\langle P_{1,1}\rangle$. That is, $E'$ is the quotient $E/\langle P_{1,1}\rangle$.  Then it is straightforward to verify that $P_1' \colonequals (\varphi(P_{1,{n+1}}))_{n}$ and $P_2' \colonequals (\varphi(P_{2,n}))_{n}$ is a basis for $T_\ell E'$.

As in the proof of \Cref{kercoker}, we compute the isomorphism $\delta \colon H \xrightarrow{\sim} \coker(T_\ell(\varphi))$, both groups isomorphic to $\Z/\ell\Z$.  We lift $P_{1,1}$ to $(P_{1,2},P_{1,3}\dots) \in T_\ell E$; so $\delta(P_{1,1})$ is the class of 
$(\varphi(P_{1,{n+1}}))_{n}=P_1'\in T_\ell E'$
in $\coker(T_{\ell}(\varphi))$.  
\end{example}

As in the introduction, we simultaneously keep track of all $\ell$-adic Tate modules (where $\ell\neq p = \opchar k$) as follows.  Let 
\begin{equation}
\Zhatp \colonequals \prod_{\ell \neq p} \Z_\ell
\end{equation}
be the prime-to-$p$ profinite completion of $\Z$, and let $\Qhatp \colonequals \Zhatp \otimes_{\Z} \Q$.  We define the \defi{(prime-to-$p$) adelic Tate module} of $A$ by
\begin{equation} \label{adelicTate}
\bT A \colonequals \varprojlim_{\substack{ n \\ p \nmid n}} A[n](\kclosed) \simeq \prod_{\ell \neq p} T_\ell A.
\end{equation}
(We write $\bT A$ instead of $\bT' A$ to ease notation.)  Then $\bT A$ is a free $\Zhatp$-module of rank $2g$.  The previous constructions extend to $\bT$ in each component, in particular from an isogeny $\varphi \colon A \to A'$ we obtain a homomorphism $\bT(\varphi) \colon \bT A \to \bT A'$.  Let $\bV A \colonequals \bT A \otimes_\Z \Q$ be the \defi{(prime-to-$p$) adelic Tate representation} of $A$.  

\begin{remark}
In the adelic Tate module, if $p \neq 0$ we can also include the case $\ell=p$; but in this case, the $p$-adic Tate module in characteristic $p$ behaves quite differently---it is quite meager in comparison, due to inseparability.  One can use the Dieudonn\'e module instead: see for example Fontaine \cite[Chapitre III]{Dieudonne}.  
\end{remark}

We then obtain the following more general statement.

\begin{prop}\label{kercoker:2}
Let $\varphi\colon A\to A'$ be a separable isogeny of abelian varieties with kernel $H$.  Then $H$ is naturally isomorphic to the cokernel of $\widehat{T}(\varphi)\colon \widehat{T} A\to \widehat{T} A'$.  
\end{prop}

\begin{proof}
The statement follows from \Cref{kercoker} by comparing $\ell$-primary parts.
\end{proof}

\subsection{Duals}\label{subsec:dual}

Associated naturally to $A$ is the \defi{dual abelian variety} \(A\spcheck \colonequals \bPic^0_A\). 
More precisely, the \defi{relative Picard functor} $\underline{\Pic}_{A}$, which associates to a $k$-scheme $S$ the group $\Pic(A \times_k S)$, is represented by a group scheme $\bPic_A$ over $k$, and $\bPic^0_A$ is the connected component containing the identity (the structure sheaf $\mathscr{O}_A$) \cite[section~3]{Kleiman}.  
The group $\Pic^0(A)=\bPic_A^0(k)$ consists of those line bundles $\Lbun$ on $A$ such that $\tau_P^*\Lbun_{\kal} \simeq \Lbun_{\kal}$ for all $P \in A(\kal)$, where $\tau_P \colon A \to A$ is the translation by $P$ map, defined by $Q \mapsto Q+P$.

Over the complex numbers $k=\C$, the dual admits a concrete description.  For $A(\C) \simeq V/\Lambda$ where $V \simeq \C^g$, let 
\begin{equation}
V^* \colonequals \Hom_{\overline{\C}}(V,\C)
\end{equation}
be the $\C$-vector space of $\C$-\emph{antilinear} functionals: that is, $f \in V^*$ means $f \colon V \to \C$ is an $\R$-linear map such that $f(ax)=\overline{a}x$ for all $x \in V$ and $a \in \C$.  (Antilinear functionals come naturally out of Hermitian forms, where one component is $\C$-linear but the other is $\C$-antilinear.)  The imaginary part of the evaluation map 
\begin{equation} \label{eqn:vvC}
\begin{aligned}
V^* \times V &\to \C \\
(f,x) &\mapsto \impart f(x)
\end{aligned}
\end{equation}
defines a canonical, nondegenerate, $\R$-bilinear form.  The dual of $\Lambda$ under this pairing, namely
\begin{equation}
\Lambda^* \colonequals \{f \in V^* : \impart f(x) \subseteq \Z \} \subseteq V^* 
\end{equation}
is a lattice and 
\begin{equation} A\spcheck(\C) \simeq V^*/\Lambda^*. \end{equation}
So in simple linear algebraic terms, the dual abelian variety is obtained from the dual lattice (with respect to \eqref{eqn:vvC}).

\begin{example}
Suppose $E=V/\Lambda$ with $V=\C$ and $\Lambda=\Z\tau + \Z$, so $\omega_1=\tau$, $\omega_2=1$ is a $\Z$-basis for $\Lambda$.  Then $V^* = \C e$ where $e(x)=\overline{x}$, and it is straightforward to compute that a $\Z$-basis for $\Lambda^*$ is $\omega_1^*=e/\impart(\overline{\tau})$, $\omega_2^*=-\tau e/\impart(\overline{\tau})=\tau e/\impart(\tau)$.  
\end{example}

Put a bit more abstractly, $A\spcheck(\C)=\bPic^0_A(\C)$ is the kernel of the map 
\begin{equation} \Pic(A) \simeq H^1(A,\O^\times_A) \to H^2(A,\Z), \end{equation} 
the boundary map in the long exact sequence in cohomology induced by the exponential sequence
\[ 0 \to \Z \to \scrO_A \to \scrO_A^\times \to 1 \]
(where $\Z$ is the constant sheaf) arising from $s \mapsto \exp(2\pi i s)$ for a section $s$ of $\scrO_A$.  For further details, see Mumford \cite[section II.9]{Mumford} or Swinnerton-Dyer \cite[\S 8]{SD}.

The universal line bundle \(\mathscr P\) on \(A\times A\spcheck\) of the functor represented by $\bPic^0_A$
is called the \defi{Poincar\'e bundle}: a point $Q \in A\spcheck$ by definition gives a line bundle $\Lbun_Q$ up to isomorphism, and $\mathscr P|_{A\times \{Q\}} \simeq \Lbun_Q$. The bundle $\mathscr P$ furthermore gives rise to a natural isomorphism \cite[Corollary,~p.~132]{Mumford}
\begin{equation}
\begin{aligned}
i_A\colon A &\xrightarrow{\sim} (A\spcheck)\spcheck \\
P &\mapsto \mathscr P|_{\{P\}\times A\spcheck}.
\end{aligned}
\end{equation}

A homomorphism \(\varphi \colon A \to A'\) of abelian varieties induces a dual homomorphism \[\varphi\spcheck \colon (A')\spcheck \to A\spcheck \]
by pullback of line bundles, i.e.~\(\varphi\spcheck\colonequals \varphi^*\colon \bPic^0_{A'} \to \bPic^0_A\). If \(\varphi\) is an isogeny, then so is \(\varphi\spcheck\).  For an isogeny of complex abelian varieties $\varphi \colon V/\Lambda \to V'/\Lambda'$ so $\varphi(\Lambda) \subseteq \Lambda'$, the pullback
\begin{equation}
\begin{aligned}
\varphi\spcheck \colon (V')^*/(\Lambda')^* &\to V^*/\Lambda^* \\
f &\mapsto \varphi^*(f) = f \circ \varphi
\end{aligned}
\end{equation}
indeed gives $\varphi\spcheck((\Lambda')^*) \subseteq \Lambda^*$, and we now have $\ker \varphi\spcheck = \Lambda^*/\varphi\spcheck((\Lambda')^*)$. 

\begin{example}
For $\varphi \colon E\to E'$ an isogeny of elliptic curves, we recover the dual isogeny $\varphi\spcheck \colon \bPic^0(E') \to \bPic^0(E)$ by restricting the pullback map $Q \mapsto \sum_{P \in \varphi^{-1}(Q)} P$ \cite[\S III.6]{SilvermanAEC} (taking classes), then applying the canonical isomorphism $E \cong \bPic^0(E)$ and the same with $E'$.
\end{example}

\subsection{Weil pairing}\label{subsec:Weil}

Recall that a finite-dimensional vector space and its dual admit a tautological perfect pairing.  There is an analogous pairing as follows.

\begin{thm}[Tautological Weil pairing] \label{thm:Weil}
Let $n \in \Z_{>0}$ be coprime to the characteristic of \(k\).  Then there is a canonical perfect, bilinear pairing
\begin{equation} \label{eqn:weilpair}
\langle \cdot, \cdot \rangle_n \colon A[n] \times A\spcheck [n] \to \mu_n,
\end{equation}
compatible with the action of $\Gal_k$.  
\end{thm}

We call \eqref{eqn:weilpair} the \defi{tautological Weil pairing}.  This pairing is an expression of \emph{Cartier duality} (for torsion in abelian varieties), and it is equivalently expressed by a canonical isomorphism 
\begin{equation}
\beta_n \colon A\spcheck[n] \xrightarrow{\sim} \Hom(A[n],\mu_n),
\end{equation}

\begin{proof}[{Proof of \Cref{thm:Weil}}]
For a detailed proof and properties of the pairing, see Oda~\cite[Theorem~1.1]{Oda}.  See also Antieau--Auel \cite[\S 2.2]{AA}.  

A definition of the pairing \eqref{eqn:weilpair} follows along similar lines as in the case of elliptic curves \cite[\S III.8]{SilvermanAEC}: see Mumford \cite[p.~184--185]{Mumford} or Hindry--Silverman \cite[Exercise A.7.8]{HinSil}.  Given $Q \in A\spcheck[n](\kal)$ corresponding to $\scrL=\scrL_Q\simeq\O(D)$, we will construct a map $A[n](\kal)\to\mu_n(\kal)$ as follows. By assumption, both $\Lbun^{\otimes n}$ and $[n]^* \Lbun$ are trivial, and therefore $nD=\ddiv(f)$ and $[n]^*D=\ddiv(g)$ for some $f=f_Q,g=g_Q\in \kal(A)^\times$.  

Since
$$\ddiv(f\circ[n])=[n]^*(nD)=n([n]^*D)=n\ddiv(g)=\ddiv(g^n), $$
the functions $f\circ[n]$ and $g^n$ differ by a scalar constant.  Now for all $X\in A(\kal)$ and $P\in A[n](\kal)$, we have $(f\circ[n])(X+P)=(f\circ[n])(X)$, and therefore $g^n(X+P)/g^n(X)=1$ and $g(X+P)/g(X)$ is constant, independent of $X$ (but may depend on $Q$).  Thus, we obtain a map
\begin{equation}
\begin{aligned}
\beta_n\colon A[n](\kal) &\to \mu_n \\
P &\mapsto\frac{g(X+P)}{g(X)}
\end{aligned}
\end{equation}
(for any choice of $X \in A(\kal)$ such that $g(X)$ and $g(X+P)$ are defined and nonzero), and 
\[ \langle P,Q \rangle_n = g_Q(P)/g_Q(X+P) \]
completing the definition of the pairing.  
\end{proof}

For $\ell \neq \opchar k$, the pairing with $n=\ell^j$ is compatible with the mult\-iplication-by-\(\ell\) map, so together they yield a perfect bilinear pairing on $\ell$-adic Tate modules:
\begin{equation}
\langle \cdot, \cdot \rangle \colon T_\ell A \times T_\ell(A\spcheck) \to \mu_{\ell^\infty}(\kclosed) \simeq \Z_{\ell}.
\end{equation}
We can also adelically put these together
\begin{equation} \label{eqn:adelictate}
\langle \cdot, \cdot \rangle \colon \bT A \times \bT(A\spcheck) \to \mu_\infty'(\kclosed) \simeq \Zhatp.
\end{equation}
Given a separable isogeny $\varphi \colon A \to A'$ with $\ker \varphi \subseteq A[n]$, comparing the left- and right-kernels of the tautological Weil pairing yields a canonical perfect pairing
\begin{equation} \label{eqn:eqnphi}
\ker \varphi \times \ker \varphi\spcheck \to \mu_n.
\end{equation}

Finally, the Weil pairing is equivariant with respect to the action of \(\Gal_k \colonequals \Gal(k^{\sep}\,|\,k)\), where \(\Gal_k\) acts on \(\mu_n\) by the mod \(n\) cyclotomic character \(\varepsilon_n\) \cite[section III.8]{SilvermanAEC} (see also \Cref{lem:cyclotomic}). 
The points of $\mu_n(\kal) = \langle \zeta_n \rangle$ are the $n$th roots of unity and the cyclotomic character
 \begin{equation}
 \varepsilon_n \colon \Gal_k \to \Aut(\mu_n(\kal)) \simeq (\Z/n\Z)^\times, 
 \end{equation}
is uniquely defined by $\sigma(\zeta_n) = \zeta_n^{\varepsilon_n(\sigma)}$ for all $\sigma \in \Gal_k$. 
Then, for all $\sigma \in \Gal_k$ and all $P \in A(\kal)$ and $Q \in A\spcheck(\kal)$, we have
\begin{equation} 
\langle \sigma(P), \sigma(Q) \rangle = \langle P,Q\rangle^\sigma = \varepsilon_n(\sigma)\cdot\langle P,Q \rangle.
\end{equation}

\subsection{Polarizations}\label{subsec:polar}

If $\Lbun$ is an ample line bundle on \(A\), then the morphism 
\begin{equation} \label{eqn:phipullaA}
\begin{aligned}
\varphi_{\Lbun} \colon A &\to A\spcheck \\
P &\mapsto \tau_P^*\Lbun\otimes \Lbun^{-1}
\end{aligned}
\end{equation} 
is an isogeny.  (By contrast, note that if $\Lbun \in \Pic^0(A)$, then \eqref{eqn:phipullaA} is the zero map.)  This observation motivates the following definition.

\begin{defin}\label{def:polarization}
An isogeny $\lambda\colon A\to A\spcheck$ is a \defi{polarization} if there is a finite separable field extension $K\supset k$ and an ample line bundle $\scrL$ on $A_{K}$ so that 
the base change of $\lambda$ to $K$ is equal to \(\varphi_\Lbun\). 

A polarization \(\lambda\) is \defi{principal} if it is an isomorphism, in which case we say that \(A\) is \defi{principally polarized} by $\lambda$.
\end{defin}

Polarizations can also be identified among isogenies over the ground field $k$: they are the isogenies \(\lambda\colon A \to A\spcheck\) where the line bundle \((\id, \lambda)^*\mathscr P\) on \(A\) is ample and $\lambda$ is \defi{symmetric}, meaning that  
\(\lambda\spcheck \circ i_A = \lambda\). 
Given such an isogeny $\lambda\colon A\to A\spcheck$, the line bundle $\scrL$ as in \Cref{def:polarization} can be constructed, after a possible finite separable extension, as a bundle satisfying the relation $[2]^*\scrL\simeq ((\id,\lambda)^*\mathscr{P})^2$ \cite[Theorem~2, p.~188]{Mumford}.

Over the complex numbers, the existence of a polarization distinguishes abelian varieties among complex tori.  Indeed, an ample line bundle on $A(\C)=V/\Lambda$ is specified by a positive definite Hermitian form $H \colon V \times V \to \C$ such that $E \colonequals \impart H$ has $E(\Lambda,\Lambda) \subseteq \Z$ (more generally, see the Appell--Humbert theorem for a linear-algebraic description of line bundles on $A$).  The restriction $E|_{\Lambda} \colon \Lambda \times \Lambda \to \Z$ is alternating, and so there exists a $\Z$-basis of $\Lambda$ in which the Gram matrix is 
\begin{equation} \label{eqnD:dio}
[E]=\begin{pmatrix} 0 & D \\ -D & 0 \end{pmatrix} 
\end{equation}
where $D=\mathrm{diag}(d_1,\dots,d_g)$ is diagonal with $d_i \geq 1$.  Then $\ker \lambda \simeq (\Z/d_1\Z)^2 \oplus \dots \oplus (\Z/d_g\Z)^2$, and so $\lambda$ is principal if and only if $d_1=\dots=d_g=1$.  

\begin{example} \label{exm:princpol}
Elliptic curves are always principally polarized.  We follow  \eqref{eqn:phipullaA}.  For $\scrL=\scrO(D)$, 
\begin{equation} 
\tau_P^* \scrL \otimes \scrL^{-1} \simeq \scrO(\tau_{-P}(D)) \otimes \scrO(-D) \simeq \scrO(\tau_{-P}(D)-D). 
\end{equation}
Now let $D=[O]$ be the origin (divisor of degree $1$); then 
\[ \tau_{-P}([O])-[O]=[-P]-[O] \sim [O]-[P], \] 
since $[P]+[-P]\sim 2[O]$.  Plugging back in, 
\begin{equation}
\tau_P^* \scrL \otimes \scrL^{-1} \simeq \scrO([O]-[P]).
\end{equation} 
Unfortunately, this is the \emph{negative} of the natural isomorphism \(\kappa \colon E \xrightarrow{\sim} E\spcheck\) given by \(P \mapsto \mathscr [P]-[O]\) \cite[\S III.6]{SilvermanAEC}, and thus $\kappa$ does not define a polarization---the sign is caused by moving between line bundles and divisors.
\end{example}

\begin{example}
If $C$ is a nice (i.e., smooth, projective, geometrically integral) curve of genus $g$ over $k$ with $C(k) \neq \emptyset$ , then its Jacobian $J \colonequals \bPic^0(C)$
is principally polarized by the theta divisor \(\Theta \subset J\), the translate under the isomorphism \(J \cong \bPic^{g-1}_C\) of the image of the natural morphism \(\Sym^{g-1}C \to \bPic^{g-1}_C\).  More generally, see Milne \cite[section~1]{Milne}. 

There is a second natural morphism \(J \to J\spcheck\), which is the inverse of the pullback morphism \(j^* \colon J\spcheck \to J\) induced by the inclusion \(j \colon C \hookrightarrow J\). Again, there is a negative sign relating the two: \(\varphi_{\mathscr{O}(\Theta)}=-(j^*)^{-1}\) \cite[Lemma~6.9]{Milne}.
\end{example}

Now let $\lambda \colon A \to A\spcheck$ be a polarization.  Then we can plug it into the tautological Weil pairing, giving a (possibly degenerate) bilinear pairing on $A[n]$:
\begin{equation}\label{usual_Weil}
\begin{aligned}
\langle \cdot,\cdot \rangle_{n,\lambda} \colon A[n] \times A[n] &\to \mu_n \\
(P,Q) &\mapsto \langle P, \lambda(Q) \rangle_n
\end{aligned}
\end{equation}
Taking $\lambda$ to be the principal polarization in \Cref{exm:princpol}, we recover the usual formula for the Weil pairing \cite[section III.8]{SilvermanAEC} (noting how the sign is compensated for).

\begin{lem} \label{lem:kerweil}
The (left or right) kernel of the Weil pairing for $\lambda$ is exactly $\ker \lambda$.
\end{lem}

\begin{proof}
If $Q \in A[n]$ has 
\[ \langle P,Q \rangle_{n,\lambda}=\langle P,\lambda(Q)\rangle_n=1 \] 
for all $P \in A[n]$, then $\lambda(Q)=O$ so $Q \in \ker \lambda_0$.  The other containment is immediate.
\end{proof}

Given a polarization  $\lambda_0\colon A_0\to A_0\spcheck$ associated with an ample line bundle $\scrL$ (cf.~\Cref{def:polarization}), we may construct other polarizations.

\begin{defin}\label{def:pullback}
Let $f\colon A\to A_0$ be an isogeny.
The \defi{pullback} of $\lambda_0$ by $f$ is
the composition $f^*\lambda_0\colonequals f\spcheck\circ\lambda_0\circ f$:
\[\xymatrix{
A \ar[d]_{f} \ar[r]^{f^*\lambda_0} & A\spcheck \\
A_0 \ar[r]^{\lambda_0} & A_0\spcheck\ar[u]_{f\spcheck}
}\]
\end{defin}

The pullback $f^*\lambda_0$ is the polarization associated with the line bundle $f^*\scrL$.  

\begin{defin}\label{def:pushforward} 
Let $\varphi \colon A_0\to A$ be an isogeny and let $d \geq 1$ be minimal such that $\ker(\varphi)\subseteq \ker(d\lambda_0)$. 
A \defi{pushforward} of $\lambda_0$ under $\varphi$ is an isogeny $\varphi_* \lambda_0$ that makes the following diagram commute:
\begin{equation}\label{pushforward}
\begin{aligned}
  \xymatrix{ A_0 \ar[r]^{d\lambda_0}\ar[d]_{\varphi} &A_0\spcheck\\
A \ar[r]^{\varphi_*\lambda_0}&A\spcheck \ar[u]_{\varphi\spcheck}
  }
  \end{aligned}
\end{equation}
\end{defin}

If a pushforward exists, it is necessarily unique, justifying the notation.  
The value of $d$ 
divides the exponent~$e_\varphi$ of $\varphi$ since $A_0[e_\varphi]\subseteq \ker(e_\varphi\lambda_0)$.

General criteria for existence of the pushforward of a polarization are given in \cite[Ch.~23]{Mumford}: $\ker(\varphi)$ must be isotropic under a certain pairing (the commutator pairing of the theta group).  This pairing is bilinear and skew-symmetric, so if $\ker(\varphi)$ is cyclic, the isotropy condition is satisfied. 
Moreover, if $\lambda_0$ is a principal polarization, then the criterion simplifies: a pushforward exists if and only if $\ker \varphi$ is isotropic under the pairing  $\langle\cdot,\cdot\rangle_{d,\lambda_0}$ on $A_0[d]$ 
defined in \eqref{usual_Weil} (see Mumford \cite[(5),~p.~228]{Mumford}). 

\section{Functorial aspects of Tate modules}
\label{sec:category}

Here we introduce the functor from abelian varieties that are isogenous quotients of a fixed abelian variety $A_0$ to sublattices of an associated adelic lattice, and we prove \Cref{thm:main}.  In the next section, we will address 
common computational cases.

In \cref{guiding}, we begin by showing a motivating example of how to conjugate the Galois action of an elliptic curve to find the action on the torsion of an isogenous curve.
In \cref{tw.on} we develop a formal categorical framework for our methods and prove our main result.

Throughout, let $k$ be a field with characteristic $p \geq 0$.  

\subsection{A guiding example}\label{guiding}

We begin with a simple example computing the Galois action on the $\ell$-torsion of an elliptic curve that is isogenous to a curve whose Galois structure is known.  

Let $E$ be an elliptic curve over $k$ and let $\ell\neq p$ be a prime number.
Let $\varphi\colon E\to E'$ be a cyclic isogeny with the choice of basis $P_1,P_2\in T_{\ell}E$ as in \Cref{cyclic_isog}.  We suppose that the point $P_{1,1} \in E[\ell](k)$ generating the kernel is a $k$-rational point, so the isogeny $\varphi$ is defined over $k$ as well. 
Let $\sigma\in \Gal_k$.  Then $\sigma P_{1,1}=P_{1,1}$ and 
$\sigma P_{2,1}=b_1 P_{1,1}+d_1P_{2,1}$ for some $b_1,d_1 \in \F_\ell$ with $d_1 \neq 0$, so the action of $\sigma$ on $E[\ell]$ in this basis is given by 
$$
\begin{pmatrix}
1&b_1\\0&d_1
\end{pmatrix}\in \GL_2(\F_\ell).
$$
(We are taking the column convention; the row convention is also used \cite[Remark 2.1]{RSZB}.)  

We now compute the action of $\sigma$ on a basis for $E'[\ell]$ consisting of images of points under $\varphi$. In \Cref{cyclic_isog} we use the basis $P'_1\colonequals \{\varphi(P_{1,n+1})\}_n$, $P_2' \colonequals \{\varphi(P_{2,n})\}$ for $T_\ell E'$, so we
may take our basis for $E'[\ell]$ to be
$P'_{1,1}=\varphi(P_{1,2})$ and $P'_{2,1}=\varphi(P_{2,1})$.
Note that it is not possible to choose points in $E[\ell]$ whose images under $\varphi$ are a basis for $E'[\ell]$; we must choose at least one of the points in $E$ to be $\ell^2$-torsion. 
To determine the action of $\sigma$ on $P'_{1,1}$, we need to know the action of $\sigma$ on $P_{1,2}$.
The action of $\sigma$ on $E[\ell^2]$ is given by $\begin{pmatrix}
a_2&b_2\\c_2&d_2
\end{pmatrix}\in \GL_2(\Z/\ell^2\Z)$,
where
\[
\begin{pmatrix}
a_2&b_2\\c_2&d_2
\end{pmatrix}
\equiv 
\begin{pmatrix}
1&b_1\\ 0 &d_1
\end{pmatrix}
\pmod{\ell}. \]
If we write $P'_{2,2}\colonequals\varphi(P_{2,2})$, then we have 
\begin{equation}
\begin{aligned}
\sigma P'_{1,1}&=\varphi(\sigma P_{1,2})=\varphi(a_2 P_{1,2}+c_2 P_{2,2})=1P'_{1,1}+c_2P'_{2,2},
  \\  
\sigma P'_{2,1}&=\varphi(\sigma P_{2,1})=\varphi(b_1 P_{1,1}+d_1P_{2,1})=d_1P'_{2,1}.
\end{aligned}
\end{equation}
We know that $c_2=\ell c'$ for some $c'\in \F_{\ell}$, and 
$c_2P'_{2,2}=c' P'_{2,1}$.
The action of $\sigma$ on $E'[\ell]$ is thus given by \[
\begin{pmatrix}
1& 0\\
c'& d_1
\end{pmatrix}.    
\]
It is interesting to see how the matrix giving the action of $\sigma$ on $E'[\ell]$ differs from the one giving the action on $E[\ell]$: it is lower rather than upper triangular. Both $b_1$ and $b_2$ do not appear, but $c_2$ does have an effect. Although it is possible to find a point of $E'[\ell]$ fixed by $\sigma$, the point one might solve for has dependence on $c'$ and $d_1$, which will vary as $\sigma$ does, meaning that $E'[\ell]$ need not have any nontrivial $k$-points. 

Now that we see how that goes, we will reinterpret the above by thinking of it as obtained from a \emph{change of basis} on the $\ell$-adic modules: see \Cref{examplecyclic}.

\subsection{From isogenies to Tate modules}\label{tw.on}

Let $A_0$ be a (fixed) abelian variety over $k$ of dimension $g$; it will function like a basepoint.  Let $\varphi\colon A_0\to A$ be an isogeny. 
Given the action of the Galois group on $T_{\ell}A_0$, we may determine the action on
$T_{\ell} A$ by identifying it with a sublattice of 
$V_{\ell}A_0\colonequals T_{\ell}A_0\otimes \Q_{\ell}$. This approach also facilitates comparing the Galois actions on more than one quotient of $A_0$.  

In this section, we describe this as a functor from isogenies to sublattices.  

In order to introduce the category that keeps track of isogenies $\varphi \colon A_0 \to A$, we will use the following general categorical construction.
Given a category and a particular choice of object, we may form a new category where we restrict our attention to morphisms into (resp.\ out of) that object, called a \defi{slice} (resp.\ \defi{coslice}) category.  

\begin{example}
Let \textsf{CRing}
be the category of commutative rings under ring homomorphisms, and let $R$ be an object of $\textsf{CRing}$, i.e., a commutative ring.  Then the objects and morphisms of the coslice category of \textsf{CRing} under $R$, denoted $\textsf{CRing}^R$, are $R$-algebras and $R$-algebra homomorphisms. 
\end{example}

A natural way to form the category we are interested in is to first consider a category whose objects are abelian varieties and whose morphisms are isogenies, and then take the coslice category for $A_0$.

Let \(\calI_{\kclosed} = \text{\textsf{AbVar}}'_{\kclosed} \) be the category whose objects are abelian varieties over $\kclosed$
 and whose morphisms are isogenies with degree prime to $p$.  Then $\Gal_k$ acts on $\calI_{\kclosed}$ (i.e., its elements act as functors compatible with the usual axioms for a group), with fixed subcategory $\calI_k$, the subcategory with objects and morphisms defined over $k$.
 
The objects of the coslice category of $\calI_{\kclosed}$ under $A_0$, denoted $\cosl_{\kclosed}$, 
are isogenies whose domain is $A_0$; the morphisms of this coslice category are commuting triangles of isogenies:
\begin{equation}\label{coslicemorph}
  \vcenter{
\xymatrix@C=0.8em{
&A_0\ar[dr]^{\varphi'}\ar[dl]_{\varphi} &\\
A\ar[rr]_{\psi}& &A'.
}}
\end{equation}

\begin{rmk}\label{initial}
In a coslice category, the identity morphism on the fixed object is an initial object of the coslice category. 
\end{rmk}  

Next we define the category of sublattices, where we will compute Galois actions.

Let $\TAze$
  be the category whose objects are $\Zhatp$-lattices $T\subset \bV A_0$ containing $\bT A_0$
and whose morphisms are injective maps of lattices in $\bV A_0$ with the diagram
\begin{equation}\label{lattice_morph}
  \vcenter{
\xymatrix@C=0.8em{ &  \bT A_0
\ar@{}[rd]^(.2){}="a"^(0.9){}="b" \ar@{^{(}->} "a";"b"  
\ar@{^{(}->}[ld] & \\
T\, \ar@{^{(}->}[rd]\ar@{^{(}->}[rr]&&T'\ar@{^{(}->}[ld]\\
&\bV A_0 &
}}
\end{equation}
commuting, or equivalently containments $T \hookrightarrow T'$.  In particular, for an object $T$ the quotient $T/\bT A_0$ is finite.  
We consider only injections of lattices in $\TAze$ since, as discussed in \cref{subsec:isogenies}, isogenies induce
 injective maps between Tate modules.  

We now define a functor
\begin{equation}\label{Psi}
\Lambda \colon \cosl_{\kclosed} \to\TAze,\end{equation} 
 which shows how Tate modules of isogenous quotients of $A_0$ are identified with sublattices of $\bV A_0$.

For any $(\varphi\colon A_0\to A) \in \Ob\cosl$,
we have the injective $\Zhatp$-linear map
$\bT\varphi\colon \bT A_0\hookrightarrow \bT A$
and the isomorphism $\bV\varphi\colon\bV A_0\xrightarrow{\sim} \bV A$. 

\begin{defin} \label{defn:lambda}
We define the sublattice of $\bV A_0$ associated with $\bT A$ via $\varphi$ to be
$$\Lambda\varphi\colonequals(\bV\varphi)^{-1}(\bT A),$$ 
which is the 
image of $\bT A$ in $\bV A_0$  under
the dotted arrow in the following commutative diagram:
\begin{equation} \label{eqn:Tvarpi}
\begin{aligned}
\xymatrix@!=2.5pc{
\bT A_0 \ar@{->}[r]^{\bT\varphi} \ar@{->}[d]_{\otimes\Q} &
  \bT A \ar@{->}[d]^{\otimes\Q}
 \ar@{-->}[dl]
  \\
\bV A_0 \ar[r]^{\sim}_{\bV\varphi}  & \bV A.
}
\end{aligned}
\end{equation}
We write the $\ell$-adic part of $\Lambda\varphi$ as $\Lambda_\ell\varphi$.

For a morphism $\psi$ in $\cosl_{\kclosed}$ as in
\eqref{coslicemorph}, we define 
\begin{equation}
\Lambda\psi \colonequals (\bV \varphi)^{-1} \circ \bT \psi \circ \bV \varphi 
\end{equation}
giving a map $\Lambda\psi \colon \Lambda\varphi \to \Lambda\varphi'$. 
\end{defin}

The map $\Lambda\psi$ comes from the commutative diagram \eqref{psi.morphism}. The longer diagonal dotted arrow factors through the shorter diagonal arrow via $\bT \psi$. The map $\Lambda\psi$ is the image of $\bT\psi$ in $\bV A_0$ via the dotted arrows:
\begin{equation}\label{psi.morphism} 
\begin{aligned}
\xymatrix@!=2.8pc{
\bT A_0 \ar@{->}[r]^{\bT\varphi} \ar@{->}[d] 
\ar@/^2pc/[rr]^{\bT\varphi'}
&
\bT A \ar@{->}[r]^{\bT \psi}
\ar@{->}[d]
  \ar@{-->}[dl]&
    \bT A' \ar@{->}[d]
\ar@{-->}[dll]    
  \\
\bV A_0 \ar[r]^{\sim}_{\bV\varphi} \ar@/_2pc/[rr]_{\bV\varphi'}  & \bV A \ar[r]^{\sim}_{\bV \psi} &\bV A'
}
\end{aligned}
\end{equation}
In particular, $\Lambda\psi$ is injective and defines a morphism in the category $\TAze$.

\begin{lem}\label{thm:main-full-functor}
The association $\varphi \mapsto \Lambda\varphi$ in \Cref{defn:lambda} defines a covariant functor $\Lambda \colon \cosl_{\kclosed} \to\TAze$.
\end{lem}

\begin{proof}
We indeed get from applying $\Lambda$ an object in the category $\TAze$, since an isogeny $A_0 \to A$ induces an inclusion $\bT A_0 \hookrightarrow \bT A$ giving an inclusion $\bT A_0 = \Lambda(\id_{A_0}) \subseteq \Lambda\varphi$.  

Next we check functoriality: it follows from functoriality of the Tate module, since given
given
\[  \vcenter{
\xymatrix{
&A_0 \ar[d]^{\varphi'} \ar[dr]^{\varphi''}\ar[dl]_{\varphi} &\\
A\ar[r]_{\psi}& A' \ar[r]_{\psi'} &A''
}}
\] 
we have
\[ \Lambda(\psi' \circ \psi) = \iota^{-1} (\bT(\psi' \circ \psi)) \iota = \iota^{-1} (\bT(\psi') \circ \bT(\psi)) \iota = \Lambda\psi'\circ\Lambda\psi \]
where $\iota=\bV \varphi$.  (We could also see this by stacking commutative diagrams like \eqref{psi.morphism}.)  Clearly $\Lambda$ also preserves the identity.
\end{proof}

We are almost ready to prove our main result.  We first prove a useful statement, recalling \Cref{kercoker}.

\begin{prop}\label{prop:Psiphi}
The association
\[ H \mapsto \Lambda(\varphi_H \colon A_0 \to A_0/H) \]
defines an inclusion-preserving bijection between 
\begin{center}
the set of subgroups $H\leq A_{0}[m](\kclosed)$ 
\end{center}
and 
\begin{center}
the set of lattices $L \subseteq (1/m)(\bT A_0)$.
\end{center}
\end{prop}

\begin{proof}
Write $L_0 \colonequals \Lambda(\id_{A_0})=\bT A_0$.  The multiplication by $m$ map $A_0 \to A_0$ corresponds to the inclusion $L_0 \hookrightarrow (1/m)L_0$, and \Cref{kercoker:2} gives a natural isomorphism $A_0[m](\kclosed) \xrightarrow{\sim} (1/m)L_0/L_0$.  

Let $H \leq H' \leq A_0(\kclosed)[m]$ be subgroups.  Then we have a composition $A_0 \to A_0/H \to A_0/H' \to A_0$ with the outer map given by multiplication by $m$.  By functoriality, this gives the containments of lattices 
\[ L_0 \subseteq L \subseteq L' \subseteq (1/m)L_0 \]
where $L=\Lambda\varphi_H$ and $L'=\Lambda\varphi_{H'}$.  This shows the map is inclusion-preserving and injective.  To show it is surjective, let $L \subseteq (1/m)L_0$.  Then $L/L_0 \leq ((1/m)L_0)/L_0 \simeq A_0[m](\kclosed)$ maps to a subgroup $H \leq A_0[m](\kclosed)$, and the map $\varphi_H \colon A_0 \to A_0/H$ also has kernel $H$, so $\Lambda\varphi_H=L$.  (This map defines an inverse.)
\end{proof}

We now prove our main result, stated in the current level of generality.

\begin{thm}\label{thm:main-full}
The functor $\varphi \mapsto \Lambda\varphi$ has the following properties.
\begin{enumalph}
\item $\Lambda$ is an equivalence of categories. 
\item $\Lambda$ is equivariant with respect to the action of $\Gal_k$, and it restricts to a functor $\Lambda_k$ which is an equivalence between the category of abelian varieties that are isogenous (over $k$)
to $A_0$ and the category of $\Gal_k$-stable sublattices of $\widehat{V}A_0$.  
\item Given an isogeny $\varphi \colon A \to A_0$, there is a natural isomorphism $\ker \varphi \simeq \Lambda\varphi/\Lambda\id_{A_0} = \coker \widehat{T}\varphi$ equivariant for $\Gal_k$.
\end{enumalph}
\end{thm}

\Cref{thm:main} is the special case where the base field $k$ is a number field.  

\begin{proof}
We skip ahead and prove part (c): it is a restatement of \Cref{kercoker:2} that for any isogeny $\varphi\colon A_0\to A$, $\ker{\varphi}$ is isomorphic to $\coker{\Lambda\varphi}$.

Next, we prove part (a); equivalence of categories is implied by full faithfulness and essential surjectivity. The functor $\Lambda$ is faithful because there exists at most one morphism between any two objects in $\cosl_{k^s}$: a diagram with $\varphi$ and $\varphi'$ as in \eqref{coslicemorph} can be completed uniquely to a commuting diagram if and only if $\ker\varphi\leq\ker\varphi'$. 
For fullness, consider any objects $\varphi,\varphi'\in\cosl_{k^s}$ and a morphism 
$\Phi\colon \Lambda\varphi\to\Lambda\varphi'$ (see \eqref{lattice_morph}). The morphism $\Phi$ is then a (unique) injection of lattices $\Lambda\varphi\to\Lambda\varphi'$, which induces an injection
$\Lambda\varphi/\bT A_0\to\Lambda\varphi'/\bT A_0$. Thus by (c) (and moreover the isomorphism specified in \Cref{kercoker:2}), we have a containment $\ker\varphi\leq\ker\varphi'$ giving an isogeny $\psi$ which uniquely fills in a diagram as in \eqref{coslicemorph}. Finally, given $\psi$, the morphism $\Lambda\psi$ is a unique injection of lattices  and thus must coincide with $\Phi$.

To see that $\Lambda$ is essentially surjective, let $L \supseteq L_0 \colonequals \Lambda(\id_{A_0})$ be a $\Zhat'$-lattice.  
Then the quotient $L/L_0$ is finite of exponent say $m$, so that $L \subseteq (1/m)L_0$.  We then apply \Cref{prop:Psiphi} to get a subgroup $H \leq A[m](\kclosed)$ such that $\varphi_H \colon A \to A/H$ has $\Lambda \varphi_H=L$. 

For part (b), the functoriality of the Tate module implies equivariance under $\Gal_k$ and objects and isogeny fixed by $\Gal_k$ descend to $k$.  We retain full faithfulness, and the argument for essentially surjectivity in (a) extends as the quotient is $\Gal_k$-stable so the subgroup $H$ is also $\Gal_k$-stable, hence $\varphi_H$ is defined over $k$.  
\end{proof}

\begin{rmk}
There is a functor from $\cosl_{k^s}$ to the category of $\Gal_k$-representations that sends an object $A$ to the Galois representation $\Gal_k\to\Aut(\widehat{T}A)$ and sends isogenies to equivariant maps. Then \Cref{thm:main-full}(b) shows that this functor factors through $\Lambda$. In the next section we furthermore choose change of basis matrices for these equivariant maps, using the lattice maps defined by $\Lambda$.   
\end{rmk}

\section{Change of basis} \label{sec:changeofbasis}

In this section, we explain how to compute with the Galois representation attached to the Tate module of an isogenous abelian variety in matrix terms.  After the main definitions in \cref{basis.recipe}, we show how this works for isogenies defined by their kernel in \cref{kernel}.  We conclude in \cref{sec:polarizationsduals} by addressing polarizations and dual isogenies.

\subsection{Bases}  \label{basis.recipe}

With $A_0$ as a reference object, we choose a $\Zhatp$-basis $\beta_0$ for $L_0 \colonequals \Lambda \id_{A_0} = \bT A_0$.  This gives us a $\Qhatp$-basis for the ambient space $\bV A_0$.  Abbreviate $n \colonequals 2g$ throughout.  

For any isogeny $\varphi \colon A_0 \to A$, we may then choose a basis $\beta$ for $L \colonequals \Lambda\varphi$ and write this in terms of the basis $\beta_0$.  We write these in the columns of a matrix $M=M_\varphi \in \mathrm{M}_{n}(\Qhatp)$; in more standard linear algebra terms, $M=[\id]_{\beta}^{\beta_0}$ is the change of basis matrix for $\bV A_0$ from $\beta$ to $\beta_0$.  

A different choice of basis $\beta$ corresponds to column operations on $M$ and is therefore given by multiplying $M$ on the right by an element of $\GL_{n}(\Zhatp)$, so we obtain a unique class in $\mathrm{M}_{n}(\Qhatp) / \GL_{n}(\Zhatp)$.  

The inverse $M^{-1}=[\id]_{\beta_0}^{\beta}$ is the matrix which describes the inclusion $\bT A_0 \hookrightarrow \Lambda(\varphi)$, and in particular it has entries in $\Zhat'$.  Now, if $\sigma \in \Aut(\bT A_0)$, we obtain a matrix $[\sigma]_{\beta_0} \in \GL_{n}(\Zhatp)$ which describes its action on the basis $\beta$; we then have
\begin{equation}
\vcenter{
\xymatrix{
\bT A_0 \ar[r]^{\sigma} \ar[d] & \bT A_0 \ar[d] \\
\Lambda\varphi \ar[r]^{\sigma} & \Lambda\varphi \\
}}
\end{equation}
and thus
\begin{equation} \label{eqn:sigmabeta}
[\sigma]_\beta = [\id]_{\beta_0}^{\beta} [\sigma]_{\beta_0} [\id]_{\beta}^{\beta_0} = M^{-1} [\sigma]_{\beta_0} M.
\end{equation}

The \emph{Hermite normal form} (an integral echelon form) gives a unique choice for the matrix $M$. For convenience, we rescale $M$ minimally by a positive integer writing $m M= M'$ with now $M' \in \mathrm{M}_n(\Zhatp)$---we then divide back through by $m$ at the end.  By column operations, we may suppose that $M'$ is lower triangular.  We may further multiply by a unique element in $\Zhatp^\times$ so that the diagonal entries $a_i$ are in $\Z_{>0}$ (rescaling the basis elements), and further that the entries in row $i$ below the diagonal are in $\{0,\dots,a_i-1\}$ (taking the elementary matrix which subtracts an appropriate multiple of the $i$th column).  The matrix $M'$ and hence $M$ is then unique with these choices.  However, this may not be the best for any individual calculation; and care must be taken so that this normalization is compatible with composition. In particular, we will not always write our change of basis matrices in this form.

This change of basis also gives us a way to work with compositions, as usual.  For example, given another isogeny $\psi \colon A \to A'$, the change of basis matrix to a basis $\beta'$ of $\bT A'$ for the composition $\psi \varphi \colon A_0 \to A'$ is obtained by multiplying the change of basis matrices:
\begin{equation}\label{eq:composition} [\id]_{\beta'}^{\beta_0} = [\id]_{\beta}^{\beta_0} [\id]_{\beta'}^{\beta}.
\end{equation}

\begin{exm}
Suppose $A=A_0$, i.e., $\varphi$ is an endomorphism of $A_0$ (with finite kernel, so still an isogeny).  In this case, we may choose $\beta=\beta_0$ and $M \in \mathrm{M}_{n}(\Zhatp)$ is the matrix of $\bT \varphi$ in the basis $\beta_0$.

In particular, the multiplication by $m \geq 1$ map on $A_0$ has matrix $\mathrm{diag}(m,\dots,m)$.  
\end{exm}

\subsection{Isogenies given in terms of kernels}\label{kernel}

We now show how to exhibit a change of basis for an isogeny in terms of its kernel.  This amounts to writing out a suitably explicit version of the connecting homomorphism $\delta$ in \Cref{kercoker}.  

We start with the fixed basis $\beta_0$ for $\bT A_0$.  Let $\varphi \colon A_0 \to A$ be an isogeny.  Let $m \in \Z_{\geq 1}$ be such that $\ker \varphi \leq A_0[m](\kclosed)$.  Let $x_1,\dots,x_r \in \bT A_0$ be elements whose reduction modulo $m$ generate $\ker \varphi$.  

\begin{lem} \label{lem:gen}
A generating set for $L=\Lambda\varphi$ is 
\[ \beta_0 \cup \{ x_1/m, \dots, x_r/m \}. \]
\end{lem}

\begin{proof}
This is just another way of writing out \Cref{prop:Psiphi}.  
\end{proof}

To get a matrix, we proceed as follows.  Let 
\[ [x_1]_{\beta_0},\dots,[x_r]_{\beta_0} \in (\Zhatp)^n \] 
be the coordinate vectors of the generators of $\ker \varphi$ in the basis $\beta_0$.  Then by \Cref{lem:gen}, a generating set is given by the columns of the identity matrix horizontally joined to the matrix with columns $(x_1/m,\dots,x_r/m)$.  

To get a basis, we then perform column echelonization of this matrix.  Concretely, we choose a minimal set of generators for $\ker \varphi$ that is lower triangular when written as a matrix---this is again possible by column operations.  It may have some zero entries along the diagonal, so we transport in the corresponding column from the identity matrix whenever a pivot is missing; we see that the remaining columns of the identity matrix are already in the span.  This matrix gives a change of basis matrix $M$ from the previous section.  

Said in an algorithmic (recursive) way, a change of basis matrix $M=(a_{ij})_{i,j}$ is determined from right to left as follows.  Dropping subscripts, let $P_{1},\dots,P_n$ be the image of the basis $\beta_{0}$ in $A_0[m](\kclosed)$.  Let $b_{n,n} \in \Z_{\geq 1}$ be the smallest positive integer such that $b_{n,n} P_n \in \ker \varphi$, and let $a_{n,n} = b_{n,n}/m$.  Let $b_{n-1,n-1} \in \Z_{\geq 1}$ be minimal such that $b_{n-1,n-1} P_{n-1} \in \ker \varphi + \langle P_n \rangle$, and then find minimal $b_{n,n-1} \in \Z/m\Z$ such that 
\begin{equation}
b_{n-1,n-1} P_{n-1} + b_{n,n-1} P_n \in \ker \varphi.
\end{equation}
We then continue this inductively and define $M=(b_{i,j}/m)_{i,j}$.  Note that the matrix $M$ produced in this way has entries in $(1/m)\Z$.

\begin{exm}\label{examplecyclic}
Let $\varphi\colon E\to E'$ be the isogeny of elliptic curves given by \Cref{cyclic_isog}
where $P_1,P_2$ is a basis for $T_\ell E$ and 
$\ker \varphi= \langle P_{1,1}\rangle$.  We take $m=\ell$.  

We have minimally $\ell P_{2,1} \in \ker \varphi$ so $b_{2,2}=\ell$.  And $P_{1,1} \in \ker\varphi$ so $b_{1,1}=1$ and $b_{2,1}=0$.  This gives the change of basis matrix
\[ M = \begin{pmatrix} 1/\ell & 0 \\ 0 & 1 \end{pmatrix}. \]
In other words, $(1/\ell)P_1,P_2$ is a basis for $\Lambda \varphi$.  

We can now describe the Galois action on $T_\ell E'$ by conjugating the action on $T_\ell E$: 
\begin{equation}
  \begin{pmatrix}
\ell&0\\ 0&1
  \end{pmatrix}
  \begin{pmatrix}
a &b\\ c&d
\end{pmatrix}      
  \begin{pmatrix}
1/\ell &0\\ 0&1
  \end{pmatrix}
=  \begin{pmatrix}
a&\ell b\\ c/\ell &d
   \end{pmatrix} 
\equiv
\begin{pmatrix}
1& 0\\ c' &d_1
   \end{pmatrix}
   \pmod{\ell}.
\end{equation}
Despite the presence of the fraction $1/\ell$ in the action we computed, this outcome is still an element of $\GL_2(\Zl)$, since $c \equiv 0 \pmod{\ell}$. As promised, this matrix conjugation recovers the action calculated by hand in \cref{guiding}.
\end{exm}

\subsection{Polarizations and duals}\label{sec:polarizationsduals}

In this section, we look at a special case where isogenies define polarizations or are dual to isogenies where we have already chosen how the functor acts.  

Let $\lambda_0 \colon A_0 \to A_0\spcheck$ be a polarization.  Then by the matrix Frobenius form, as in \eqref{eqnD:dio} we may choose a basis for $\bT A_0$ such that the induced Tate pairing $\bT A_0 \times \bT A_0 \to \Zhatp$ is of the form 
\begin{equation}\label{polarmatrix}
\begin{pmatrix}
  0  & D\\
  -D & 0
\end{pmatrix}
\end{equation}
where $D=\mathrm{diag}(d_1,\dots,d_g)$ and $d_1 \mid \dots \mid d_g$ with $d_i \in \Z_{\geq 1}$.  

\begin{lem}
The matrix \eqref{polarmatrix} defines a change of basis matrix $M^{-1}$ for $\Lambda\lambda_0$.  
\end{lem}

\begin{proof}
We compute that the inverse to the Gram matrix is 
\[ \begin{pmatrix} 0 & -D^{-1} \\ D^{-1} & 0 \end{pmatrix} \]
(where $D^{-1}=\mathrm{diag}(1/d_1,\dots,1/d_g)$).  
To prove the lemma, we need to show that
\[ -(1/d_1)P_{g+1},\dots,-(1/d_g)P_{2g}, (1/d_1)P_1,\dots,(1/d_g)P_g \]
is a basis for $\Lambda\lambda_0$.  

Let $m=d_g$, so that $\ker \lambda_0 \leq A_0[m]$ (since $d_1 \mid \dots \mid d_g$).  By \Cref{prop:Psiphi}, as elaborated upon in the previous section, it is equivalent to show that 
\[ -(m/d_1)P_{g+1,m},\dots,-(m/d_g)P_{2g,m}, (m/d_1)P_{1,m},\dots, (m/d_g)P_{g,m} \]
is a minimal generating set for $\ker \lambda_0$.  These elements generate a subgroup of size $(d_1\cdots d_g)^2=\#\ker \lambda_0$, so it suffices to show that they belong to the kernel.  

Recall from \Cref{lem:kerweil} that the kernel of the Weil pairing for $\lambda_0$ is $\ker \lambda_0$.  With this, the verification is straightforward: e.g., for $i \leq g$,
\[ \langle P_j, (m/d_i)P_i \rangle_{m,\lambda_0} \equiv 0 \pmod{m} \]
for all $j$: if $j \neq i+g$ then we get zero, otherwise we get $- d_i (m/d_i) \equiv 0 \pmod{m}$.
\end{proof}

\begin{exm}\label{pol_type}
  Let $\lambda_0\colon A_0\to A_0\spcheck$ be a polarization with type $(d_1,\ldots,d_g)$ and let $\ell\neq p$ a prime.
  Let $D_{\ell}\colonequals \mathrm{diag}(\ell^{n_1},\ldots, \ell^{n_g})$, where $\ell^{n_i}$ is the highest power of $\ell$ dividing $d_i$.
  Choosing a basis for $T_\ell A_0$ as explained above, the change of basis matrix for the $\ell$-adic part of $\lambda_0$ is
 \[M_{\lambda_0, \ell} = \begin{pmatrix}
  0  & -D_\ell^{-1}\\
  D_\ell^{-1} & 0
   \end{pmatrix}.\]
 If $\lambda_0$ is a principal polarization on an elliptic curve, then for any $\ell\neq p$, we get
 $$M_{\lambda_0,\ell}=\begin{pmatrix}0&-1\\1&0
  \end{pmatrix}.$$
If $\lambda_0$ is a $(1,3)$-polarization on an abelian surface (we will see $(1,\ell)$-polarizations in \cref{example.section}), the Gram matrix for the Tate pairing, and hence $M_{\lambda_0,\ell}^{-1}$, is
\[\begin{pmatrix}
    0&0&1&0\\
    0&0&0&3\\
    -1&0&0&0\\
    0&-3&0&0
    \end{pmatrix}.\]
\end{exm}  

\medskip

Next, let $\varphi\colon A\to B$ and suppose we have a change of basis matrix $M_{\varphi}$ associated to $\varphi$. 
We may identify $\bV A\spcheck, \bV B\spcheck$ with the dual vector spaces of $\bV A, \bV B$; the dual of a basis for $\bT A\subset \bV A$ gives a basis for $\bT A\spcheck$.
With these choices, the change of basis matrix $M_{\varphi\spcheck}$ for $\varphi\spcheck\colon B\spcheck \to A\spcheck$ is the transpose of $M_{\varphi}$. 

In the following example, we compute the pullback and pushforward of a principal polarization on an elliptic curve. 

\begin{exm}\label{ex:pushforwardpol}
Let $\varphi\colon E\to E'$ be the cyclic isogeny introduced in \Cref{cyclic_isog}; we computed $M_{\varphi}$ in \Cref{examplecyclic}. Choose a basis for $T_\ell E'$ so that, given the principal polarization $\lambda_0\colon E'\to {E'}\spcheck$, the change of basis matrix is $M_{\lambda_0}=\bigl(\begin{smallmatrix}0&-1\\1&0
        \end{smallmatrix}\bigr)$, as shown in \Cref{pol_type}.

To compute the pullback of $\lambda_0$ by $\varphi$, we use the equality
$\varphi^*\lambda_0=\varphi\spcheck\circ \lambda_0\circ \varphi$ from \Cref{def:pullback} along with \eqref{eq:composition}, and so the change of basis matrix $M_{\varphi^*\lambda_0}$ is given by the following matrix:
\[ M_{\varphi}M_{\lambda_0}M_{\varphi\spcheck} = \begin{pmatrix}
1/\ell&0  \\0&1\end{pmatrix}
 \begin{pmatrix}
   0&-1  \\ 1&0\end{pmatrix}
 \begin{pmatrix}
   1/\ell&0  \\0&1\end{pmatrix}^T
 = \begin{pmatrix}
   0&-1/\ell  \\1/\ell&0\end{pmatrix}.
\]
This means that the pullback $\varphi^*\lambda_0$ is $\ell$ times the principal polarization on $E$ (recalling that the Gram matrix for the polarization is given by $M_{\varphi^*\lambda_0}^{-1}$).

Next, suppose $\lambda_0 \colon E \to E\spcheck$ is the principal polarization on $E$. We may compute the pushforward of
$\lambda_0$ by $\varphi$
since the kernel of $\varphi$ is isotropic under the pairing given by $\lambda_0$.
As $\ell$ is the smallest value so that $\ker(\varphi)\subseteq \ker(\ell\lambda_0)$, we have the following equality from~\Cref{def:pushforward}:
$\ell\lambda_0=\varphi\spcheck\circ \varphi_*\lambda_0\circ  \varphi$. This means $M_{\varphi_*\lambda_0}$ is given by the following matrix:
\[ M_{\varphi}^{-1}M_{\ell\lambda_0}M_{\varphi\spcheck}^{-1}=
\begin{pmatrix}
{\ell}&0  \\0&1
\end{pmatrix}  
\begin{pmatrix}
0&-1/\ell  \\1/\ell&0
\end{pmatrix}
\begin{pmatrix}
{\ell}&0  \\0&1
\end{pmatrix}
=
\begin{pmatrix}
0&-1  \\ 1 &0
\end{pmatrix}.
\]
Thus, the pushforward $\varphi_*\lambda_0$ is a principal polarization on the elliptic curve $E'$.
\end{exm}

\section{An extended example}\label{example.section}

In this section, we give an extended example demonstrating how to use the technical tools developed in sections \ref{sec:category}-\ref{sec:changeofbasis} to compute the Galois action on abelian surfaces which are isogenous to a fixed abelian surface. We start with the construction of an abelian surface $A$, and then we will compute the Galois action on both \(A[\ell]\) and on \(A\spcheck [\ell]\) by comparing \(T_\ell A\) and \(T_\ell A\spcheck\) inside \(V_\ell A_0\) for \(A_0\) a product of elliptic curves isogenous to both \(A\) and \(A\spcheck\). 
\medskip

Fix a prime $\ell$, and let \(E_1\) and \(E_2\) be elliptic curves over \(\Q\) with $P \in E_1[\ell](\Q)$ and $Q \in E_2[\ell](\Q)$  $\Q$-rational $\ell$-torsion points.  Let
\[ G \colonequals \langle (P,Q) \rangle \leqslant E_1 \times E_2
\quad\text{and}\quad
A\colonequals (E_1\times E_2)/G, \]
with the quotient map $q \colon E_1 \times E_2 \to A$. We will consider $A$ with the polarization $\lambda$ which is the pushforward under $q$ (cf.~\Cref{def:pushforward}) of the principal product polarization $\lambda_0$ on $E_1 \times E_2$.  In \cite[Lemma~2.1.2]{FHV_research}, we show that $\lambda$ is a $(1,\ell)$-polarization, although we will not need that fact here. 

This construction is considered in \cite[Construction~2.1.1]{FHV_research}, although that is not the first appearance of such a surface in the literature; it is described on MathOverflow \cite{MO}, implicitly suggested as an exercise \cite[Exercise 6.35]{Goren}, and recently exhibited \cite[Theorem 2.5]{BS}.

\subsection{Computing the Galois action on \texorpdfstring{$A$}{A}}\label{sec:GalactionA}

To understand the action of the Galois group on $A[\ell]$, we use the image of the Galois action on $T_\ell({E}_1\times {E}_2)$, along with a change of basis matrix for \(\Lambda q\subset V_\ell({E}_1\times {E}_2)\), as explained in \cref{sec:changeofbasis}. 

In our example of interest (both here and in \cref{sec:GalactionAdual2}), the isogenies all have degrees which are powers of $\ell$, so we need only consider the $\ell$-adic portion of the Tate modules in question.  In particular, we will be interested in the mod $\ell$ representation, which we obtain by reducing modulo $\ell$ the Tate module.

For any elliptic curve \(E\) with a rational point $P\in E[\ell](\Q)$, let $P_1,P_2$ be a basis for $T_\ell(E)$ such that $P_1 \bmod \ell = P$. Then the image of the $\ell$-adic Galois representation
\[ \rho_{E,\ell}\colon \Gal_{\Q} \to \Aut(T_{\ell}(E)(\Qbar)) \simeq \GL_2(\Z_\ell) \]
is contained in 
\begin{equation} \label{eqn:surjEimg}
\left\{\begin{pmatrix}
 a & b \\ \ell c & d 
 \end{pmatrix} \in \mathrm{M}_2(\Z_\ell) : a,d \in \Z_\ell^\times, a \equiv 1 \bmod \ell \right\}\leqslant \GL_2(\Z_\ell).
 \end{equation}
Moreover, reducing modulo $\ell$, the representation 
\[ \overline{\rho}_{E,\ell}\colon \Gal_{\Q} \to \Aut({E}[\ell](\Qbar)) \simeq \GL_2(\F_\ell) \]
has image contained in
\[\left\{\begin{pmatrix}
 1 & b \\ 0 & d 
 \end{pmatrix} \in \mathrm{M}_2(\F_\ell) : d \in \F_\ell^\times \right\}\leqslant \GL_2(\F_\ell). \]

With this observation in mind, we choose a basis $\{P_1, P_2, Q_1, Q_2\}$ for $T_\ell(  E_1 \times E_2 )\simeq \Z_\ell^4$ that satisfies the following:
\begin{itemize}
    \item $P_1 \bmod \ell =P  \in E_1[\ell](\Q)$,
    \item $Q_1 \bmod \ell =Q \in E_2[\ell](\Q)$,
    \item $\{P_1,P_2\}$ is a symplectic basis for $T_\ell E_1$, and
    \item $\{Q_1,Q_2\}$ is a symplectic basis for $T_\ell E_2$.
\end{itemize}
Then the Galois action on $( E_1\times E_2)[\ell](\Qbar)$ has image contained in the subgroup
\begin{equation}\label{eq:thebigsubmodell}
\left\{ \begin{pmatrix} 1 & b_1 & 0 & 0 \\ 0 & d_1 & 0 & 0 \\ 0 & 0 & 1 & b_2 \\ 0 & 0 & 0 & d_2 \end{pmatrix} \in \mathrm{M}_4(\F_\ell) : a_1, d_1, a_2, d_2 \in \F_\ell^\times \right\}\leqslant \GL_4(\F_{\ell}).
\end{equation}

In fact, there is a further condition on elements in the image of $\bar{\rho}_{E_1\times E_2, \ell}$, determined by the Galois equivariance of the Weil pairing. We summarize this in the following lemma.

\begin{lem}\label{lem:cyclotomic}
For any elliptic curve \(E\) over \(\Q\) and points \(P,Q\in E[\ell](\Qbar)\), the cyclotomic character $\varepsilon_\ell\colon \Gal_{\Q} \to  \Z_\ell^{\times}$ satisfies \(\langle \bar\rho_{E,\ell}(\sigma)P, \bar\rho_{E,\ell}(\sigma)Q\rangle = \varepsilon_\ell(\sigma)\cdot \langle P,Q\rangle\),
for all \(\sigma \in \Gal_\Q\), where \(\langle \cdot, \cdot \rangle\) is the Weil pairing. Moreover, this implies that if \(\bar\rho_{E,\ell}(\sigma)=M\in \GL_2(\F_\ell)\), then \(\varepsilon_\ell(\sigma)=\det M\).
\end{lem}
\begin{proof}
The first claim follows directly from the Galois equivariance of the Weil pairing \cite[section III.8]{SilvermanAEC}. For the second statement, let \(\{P_1, P_2\}\) be a symplectic basis for \(E[\ell](\Qbar)\), so that, as in \Cref{pol_type}, the Gram matrix for the Weil pairing is 
\[B\colonequals \begin{pmatrix} 0 & 1 \\ -1 & 0\end{pmatrix}.\]
Then for points \(P=a_1P_1+a_2P_2\) and \(Q=b_1P_2+b_2P_2\) in \(E[\ell](\Qbar)\), we have
\[\langle \bar\rho_{E,\ell}(\sigma)P, \bar\rho_{E,\ell}(\sigma)Q\rangle = \begin{pmatrix}a_1 & a_2 \end{pmatrix} M^T B M \begin{pmatrix} b_1 \\ b_2 \end{pmatrix}=\det M\cdot \begin{pmatrix}a_1 & a_2 \end{pmatrix} B \begin{pmatrix} b_1 \\ b_2 \end{pmatrix} = \det M \cdot \langle P,Q\rangle.
\]
Thus, \(\varepsilon_\ell(\sigma)=\det M\).
\end{proof}
Consequently, for our elliptic curves \(E_1\) and \(E_2\), \(\det \bar\rho_{E_1,\ell}(\sigma) = \det \bar\rho_{E_2,\ell}(\sigma)\) for all \(\sigma \in \Gal_\Q\), so $d_1=d_2$.
\Cref{lem:cyclotomic} also holds for any $\ell^n$-torsion points (or more generally on $T_\ell(E)$), and this implies that \(\rho_{E_1\times E_2,\ell}(\Gal_\Q)\) is contained in
\begin{equation}\label{eq:thebigsub}
G_\ell \colonequals \left\{ \begin{pmatrix} a_1 & b_1 & 0 & 0 \\ \ell c_1 & d_1 & 0 & 0 \\ 0 & 0 & a_2 & b_2 \\ 0 & 0 & \ell c_2 & d_2 \end{pmatrix}\in \mathrm{M}_4(\Z_\ell) : 
\begin{minipage}{27ex}
$a_1, d_1,a_2,d_2 \in \Z_\ell^\times$, \\ 
$a_1 \equiv a_2 \equiv 1 \bmod{\ell}$, and \\
$a_1d_1-\ell b_1c_1 = a_2d_2-\ell b_2c_2$
\end{minipage}
\right\}\leqslant \GL_4(\Z_{\ell}). 
\end{equation}

For convenience, we rewrite the elements in $G_\ell$ as
\begin{equation}\label{eq:matrixl2}
\begin{pmatrix} 1+x_1\ell & b_1+y_1\ell & 0 & 0 \\ w_1\ell & d+z_1\ell & 0 & 0 \\ 0 & 0 & 1+x_2\ell & b_2+y_2\ell \\ 0 & 0 & w_2\ell & d+z_2\ell \end{pmatrix}= \begin{pmatrix} A_1 & 0 \\ 0 & A_2 \end{pmatrix}
\end{equation}
where:
\begin{itemize}
\item $d \in \{1,\dots,\ell-1\}$,
\item $b_1,b_2 \in \{0,\dots,\ell-1\}$, and
\item $w_i, x_i, y_i, z_i \in \Z_\ell$
\end{itemize}
still subject to the condition (\Cref{lem:cyclotomic}) that
\begin{equation}\label{detcond} 
\det A_1 =\det A_2. 
\end{equation}

Now, we follow the discussion in~\cref{sec:changeofbasis} to write down the change of basis matrix for $\Lambda_\ell q \subseteq V_\ell ({E}_1\times {E}_2)$. 
Recall that $A=(E_1\times E_2)/\langle(P,Q)\rangle$, so we can take $\{P_1, P_2, \frac{1}{\ell}(P_1+Q_1), Q_2\}$ as a basis for $\Lambda_\ell q$. 
Then the change of basis matrix $M_{q}$ 
is given by
\begin{equation}\label{eq:Mq}
  M_{q}=\begin{pmatrix} 1 & 0 & 1/\ell & 0 \\ 0 & 1 & 0 & 0 \\ 0 & 0 & 1/\ell & 0 \\ 0 & 0 & 0 & 1\end{pmatrix}.  
\end{equation}
This choice does not follow the algorithmic approach given after~\Cref{lem:gen} for finding the change of basis matrix, but it is a straightforward exercise to see how this choice differs from that by column operations. 
As in \eqref{eqn:sigmabeta}, to understand the Galois action on $A[\ell](\Qbar)$, we conjugate the elements \eqref{eq:matrixl2} above by $M_{q}$, 
which gives
\begin{equation}\label{eq:Amodl2matrices}
\begin{pmatrix} 1+x_1\ell & b_1+y_1\ell & x_1-x_2 & -b_2-y_2\ell \\ w_1\ell & d+z_1\ell & w_1 & 0 \\ 0 & 0 & 1+x_2\ell & b_2\ell+y_2\ell^2 \\ 0 & 0 & w_2 & d+z_2\ell \end{pmatrix},
\end{equation}
with the same conditions on the variables.  To get the image of $\bar\rho_{A,\ell}\colon \Gal_{\Q} \to \GL_4(\F_\ell)$, we reduce this subgroup modulo $\ell$. That is, the image of $\bar\rho_{A,\ell}$ is contained in the subgroup
\[\left\{ \begin{pmatrix} 1 & b_1 & x_1-x_2 & -b_2 \\ 0 & d & w_1 & 0 \\ 0 & 0 & 1 & 0 \\ 0 & 0 & w_2 & d \end{pmatrix} \in \mathrm{M}_4(\F_\ell) : 
\begin{minipage}{13ex}
$d \in \F_\ell^{\times}$\\
$b_i,w_i,x_i \in\F_\ell$
\end{minipage} \right\}\leqslant \GL_4(\F_\ell).\]  

\subsection{Computing the Galois action on \texorpdfstring{$A\spcheck$}{Av}}\label{sec:GalactionAdual2}

Next, we compute the Galois action on $A\spcheck[\ell](\Qbar)$, again using the framework developed in sections \ref{sec:category}-\ref{sec:changeofbasis}.
To do this, we will use the isogeny between $A$ and $A\spcheck$
given by the $(1,\ell)$-polarization $\lambda$ on $A$
to relate their Galois representations. 

First, as mentioned in \cref{sec:GalactionA}, the polarization $\lambda$ on $A$ is the pushforward of the principal polarization $\lambda_0$ on $E_1\times E_2$
by the quotient isogeny $q$ (cf.~Definition~\ref{def:pushforward}), as shown in the following commutative diagram:
\begin{equation}\label{pushq}
\vcenter{
\xymatrix{
E_1\times E_2 \ar[r]^-{\ell\lambda_0} \ar[d]^{q} & (E_1\times E_2)\spcheck \\
A \ar[r]^{\lambda} \ar[r]^{\lambda} & A\spcheck \ar[u]_{q\spcheck}.
}
}
\end{equation}
This commutative diagram allows us to directly compare the actions of the Galois group on $T_{\ell}A$ and $T_{\ell}A\spcheck$ as sublattices of $V_{\ell}(E_1\times E_2)$. If $M_\lambda$ is the change of basis matrix from a basis of $T_\ell A\spcheck$ in $V_\ell A$, we must conjugate the elements \eqref{eq:Amodl2matrices}, which describe the Galois action on $T_\ell A$, by $M_\lambda$. Alternatively (and equivalently), the change of basis matrix associated to the composition $\lambda q$ is given by $M_qM_\lambda$, as in \eqref{eq:composition}, so we could conjugate the elements \eqref{eq:matrixl2}, which describe the Galois action on $T_\ell(E_1\times E_2)$, by this product.   

To find a change of basis matrix for $\lambda$ which is compatible with the bases already chosen for $T_\ell(E_1\times E_2)$ and $T_\ell A$, we use the equality $\ell\lambda_{0}=q\spcheck\circ \lambda\circ q$ from \eqref{pushq}, which means $M_\lambda = M_q^{-1}M_{\ell \lambda_0}(M_q^T)^{-1}$. Following the discussion about change of basis matrices for polarizations in \cref{sec:polarizationsduals} (cf.~\Cref{ex:pushforwardpol}), we find 
\[M_{\ell \lambda_0}=\begin{pmatrix} 0 & -1/\ell & 0 & 0 \\ 1/\ell & 0 & 0 & 0 \\ 0 & 0 & 0 & -1/\ell \\ 0 & 0 & 1/\ell & 0
\end{pmatrix}.\]
Thus, 
\[
M_{\lambda}=\begin{pmatrix} 1 & 0 & -1 & 0 \\ 0 & 1 & 0 & 0 \\ 0 & 0 & \ell & 0 \\ 0 & 0 & 0 & 1\end{pmatrix}
\begin{pmatrix} 0 & -\frac1\ell & 0 & 0 \\ \frac1\ell & 0 & 0 & 0 \\ 0 & 0 & 0 & -\frac1\ell \\ 0 & 0 & \frac1\ell & 0
\end{pmatrix}
\begin{pmatrix} 1 & 0 & -1 & 0 \\ 0 & 1 & 0 & 0 \\ 0 & 0 & \ell & 0 \\ 0 & 0 & 0 & 1\end{pmatrix}^T
=
\begin{pmatrix}
0&-\frac1\ell &0 &\frac1\ell \\
\frac1\ell&0&0 & 0\\
0&1 &0 & -1\\
-\frac1\ell & 0 & 1 &0
\end{pmatrix}.
\]
One can check that the cokernel of $M_\lambda^{-1}$, which we recall from
\Cref{kercoker} is isomorphic to the kernel of \(\lambda\),
has $\ell^2$ elements, confirming that the polarization type of $\lambda$ is $(1,\ell)$.
Finally, conjugating the Galois action on $T_\ell A$ by $M_\lambda$ gives the subgroup 
\[\left\{ M_\lambda^{-1} \begin{pmatrix} 1+x_1\ell & b_1+y_1\ell & x_1-x_2 & -b_2-y_2\ell \\ w_1\ell & d+z_1\ell & w_1 & 0 \\ 0 & 0 & 1+x_2\ell & b_2\ell+y_2\ell^2 \\ 0 & 0 & w_2 & d+z_2\ell \end{pmatrix} M_\lambda\in\mathrm{M}_4(\Z_\ell) : 
\begin{minipage}{16ex}
$d\in \F_\ell^{\times}$,\\ $b_i\in \F_\ell$,\\ $w_i,x_i, y_i, z_i \in \Z_\ell$ 
\end{minipage}
\right\} \]
\[=\left\{ \begin{pmatrix}
d + z_1\ell & -w_1\ell &  0 &  0 \\
-b_1 - y_1\ell & 1 + x_1\ell & 0 & 0 \\
z_1 - z_2 & -w_1 & d + z_2\ell & -w_2 \\
 b_2 + y_2\ell &  0 &  -b_2\ell - y_2\ell^2 & 1 + x_2\ell   
\end{pmatrix}\in \mathrm{M}_4(\Z_\ell):
\begin{minipage}{16ex}
$d\in \F_\ell^{\times}$,\\ $b_i\in \F_\ell$,\\ $w_i,x_i, y_i, z_i \in \Z_\ell$ 
\end{minipage}
\right\}\]
in \(\GL_4(\Z_\ell)\). This subgroup reduces $\bmod\,\ell$ to the subgroup
\[\left\{ \begin{pmatrix}
d  & 0 &  0 &  0 \\
-b_1  & 1  & 0 & 0 \\
z_1 - z_2 & -w_1 & d  & -w_2 \\
 b_2  &  0 &  0 & 1    
\end{pmatrix}\in\mathrm{M}_4(\F_\ell) : 
\begin{minipage}{13ex}
$d \in \F_\ell^{\times}$\\
$b_i,w_i,z_i\in \F_\ell$
\end{minipage}
\right\}\leqslant \GL_4(\F_\ell).\]
Thus, the image of the mod $\ell$ representation $\bar\rho_{A\spcheck,\ell}\colon \Gal_{\Q} \to \Aut({A\spcheck}[\ell](\Qbar)) \simeq \GL_4(\F_\ell)$ is contained in this subgroup.

\begin{rmk}
    The Galois action on $A\spcheck$ can also be computed using the fact that $\rho_{A\spcheck, \ell}$ is the contragredient representation to $\rho_{A,\ell}$ twisted by the cyclotomic character \cite[Lemma~2.4.1]{FHV_research}. That result, given in \cite[Proposition~2.4.2]{FHV_research}, precisely matches this computation. The following determinantal condition \eqref{detcond} from the cyclotomic character is necessary to compare the two:
    \[
z_1-z_2 = b_1w_1-b_2w_2-dx_1+dx_2 \in \F_\ell.
    \]
\end{rmk}

%% ------%%%%%%%%%%%%%%%%%%%%%%%%%%%%%%%%%%%%%%%%%%%%%%%%%%%%

\end{document}